\documentclass[12pt,a4paper,twoside]{amsart}

\numberwithin{equation}{section}
\usepackage{rotating}
\usepackage{mathrsfs}
\usepackage{amscd,amssymb,amsopn,amsmath,amsthm,graphics,amsfonts,enumerate,verbatim,calc}
\usepackage[all]{xy}
\usepackage[utf8]{inputenc}
\usepackage{color}

\newtheorem{theorem}{Theorem}[section]
\newtheorem{claim}[theorem]{Claim}

\newtheorem{proposition}[theorem]{Proposition}
\newtheorem{fact}[theorem]{Fact}

\newtheorem{observation}[theorem]{Observation}
\newtheorem{corollary}[theorem]{Corollary}

\theoremstyle{definition}
\newtheorem{definition}[theorem]{Definition}

\newtheorem{problem}[theorem]{Problem}

\newtheorem{discussion}[theorem]{Discussion}

\newtheorem{convention}[theorem]{Convention}

\newtheorem{hypothesis}[theorem]{Hypothesis}

\theoremstyle{remark}
\newtheorem{remark}[theorem]{Remark}
\newtheorem{question}[theorem]{Question}
\newtheorem{notation}[theorem]{Notation}

\newcommand{\Ker}{{\rm Ker}}

\newcommand{\fr}{{\rm fr}}

\newcommand{\add}{{\rm add}}

\newcommand{\ZFC}{{\rm ZFC}}
\newcommand{\Eend}{{\rm End}}

\newcommand{\sfr}{{\rm sfr}}
\newcommand{\sgr}{{\rm sgr}}

\newcommand{\Ord}{{\rm Ord}}

\newcommand{\BB}{{\rm BB}}

\newcommand{\ord}{{\rm ord}}
\newcommand{\Hom}{{\rm Hom}}

\newcommand{\bd}{{\rm bd}}

\newcommand{\gr}{{\rm gr}}

\newcommand{\id}{{\rm id}}

\newcommand{\tdu}{TDU}

\newcommand{\Rang}{{\rm Im}}

\newcommand{\rest}{{\restriction}}
\newcommand{\dom}{{\rm dom}}

\newcommand{\rng}{{\rm Im}}

\newcommand{\Wilog}{{\rm Without loss of generality}}

\newcommand{\then}{{\underline{then}}}
\newcommand{\when}{{\underline{when}}}

\newcommand{\Iff}{{\underline{iff}}}
\newcommand{\mn}{{\medskip\noindent}}
\newcommand{\sn}{{\smallskip\noindent}}

\newcommand{\cF}{{\mathscr F}}

\newcommand{\bbN}{{\mathbb N}}
\newcommand{\cM}{{\mathscr M}}

\newcommand{\bbZ}{{\mathbb Z}}

\newcommand{\gx}{{\mathfrak x}}
\newcommand{\cU}{{\mathscr U}}

\newcommand{\cf}{{\rm cf}}
\newcommand{\pr}{{\rm pr}}

\newcount\skewfactor
\def\mathunderaccent#1#2 {\let\theaccent#1\skewfactor#2
\mathpalette\putaccentunder}
\def\putaccentunder#1#2{\oalign{$#1#2$\crcr\hidewidth
\vbox to.2ex{\hbox{$#1\skew\skewfactor\theaccent{}$}\vss}\hidewidth}}

\begin{document}

\title[trivial duality conjecture ] {Quite free $p$-groups with trivial duality}

\author[M. Asgharzadeh]{Mohsen Asgharzadeh}

\address{Mohsen Asgharzadeh, Hakimiyeh, Tehran, Iran.}

\email{mohsenasgharzadeh@gmail.com}

\author[M.  Golshani]{Mohammad Golshani}

\address{Mohammad Golshani, School of Mathematics, Institute for Research in Fundamental Sciences (IPM), P.O.\ Box:
	19395--5746, Tehran, Iran.}

\email{golshani.m@gmail.com}

\author[S. Shelah]{Saharon Shelah}

\address{Saharon Shelah, Einstein Institute of Mathematics, The Hebrew University of Jerusalem, Jerusalem,
	91904, Israel, and Department of Mathematics, Rutgers University, New Brunswick, NJ
	08854, USA.}

\email{shelah@math.huji.ac.il}
\thanks{The second author’s research has been supported by a grant from IPM (No. 1403030417). The
second author’s work is based upon research funded by Iran National Science Foundation (INSF) under
project No. 4027168. The third author research partially supported by the Israel Science Foundation
(ISF) grant no: 1838/19, and Israel Science Foundation (ISF) grant no: 2320/23; Research partially
supported by the grant “Independent Theories” NSF-BSF, (BSF 3013005232). This is 1045a in Shelah’s
list of publications.}





\begin{abstract}
We present a class of abelian groups that exhibit a high degree of freeness while possessing no non-trivial homomorphisms to a canonical free object. Unlike prior investigations, which primarily focused on torsion-free groups, our work broadens the scope to include groups with torsion.
Our main focus is on  $p$-groups, for which we formulate and prove the \emph{Trivial Duality Conjecture}.
Key tools in our analysis include the multi black box method and the application of specific homological properties of relative trees.
\end{abstract}

	\date{\today}
\maketitle
\medskip

\section {Introduction} \label{0}

This paper addresses the \emph{Trivial Duality Conjecture}, mainly for torsion abelian groups. Specifically, we are concerned with the following folklore problem and its innovative resolution:
\begin{problem}
	Given an infinite cardinal $\mu$, is there a $\mu$-free abelian group $G$ such that $\Hom(G, \mathbb{Z}) = 0$?
\end{problem}

We denote this \emph{trivial dual property} by $\tdu_{\mu} $ when $\mu > \aleph_0$. There are a lot of works over abelian groups. Here  is a short list. Recall that much earlier results, like the existence of an $\aleph_1$-free abelian group $G$ of cardinality $\aleph_1$ with $\Hom(G, \mathbb{Z}) = 0$, were established by Eda \cite{eda} and Shelah \cite{Sh:105}.
 The existence of such groups was known classically for $\aleph_1$-free abelian groups, but remained widely open for many years for $\aleph_n$-free abelian groups, where $n > 1$. This was finally answered affirmatively in \cite{Sh:883}, where examples using $n$-dimensional black boxes were introduced. Subsequently, these were used for more complicated algebraic relatives in G\"obel-Shelah \cite{Sh:920}.
In \cite{Sh:898}, Shelah introduced several close approximations to proving in ZFC some almost positive results for $\aleph_\omega$, that is $\tdu_{\aleph_\omega}$, using 1-black boxes. In  his landmark paper \cite{Sh:1028}, Shelah finally proved that $\tdu_{\aleph_{\omega}}$, and indeed $\tdu_{\aleph_{\omega_1 \cdot k}}$ holds for all $k < \omega$. Furthermore, assuming the existence of large cardinals, he showed that $\tdu_{\aleph_{\omega_1 \cdot \omega}}$ can consistently fail. This demonstrates that $\lambda = \aleph_{\omega_1 \cdot \omega}$ is the first cardinal for which $\tdu_\lambda$ cannot be proved in ZFC.
Despite a lot of works over abelian groups, the trivial duality problem was largely restricted to torsion-free groups. This inspires us to continue exploring around Problem 1.1. In particular, we address the following natural problem:

\begin{problem}\label{p1.2}
How can one extend $\tdu $ to not necessarily torsion-free groups?
\end{problem}

The structure theorem for torsion groups states that every torsion group can be uniquely decomposed into a direct sum of its $p$-primary components for each prime $p$, where each $p$-primary component consists of elements whose orders are powers of $p$. This means that to understand a torsion group, we can focus on understanding $p$-groups (groups where each element has
an order that is a power of $p$) for different primes $p$.

\begin{hypothesis}
Assume $0<\mathbf{k}<\omega$ and
let  $\bar\partial=\langle \partial_\ell=\partial(\ell):\ell
< \bold k\rangle$ be a sequence of regular cardinals. \begin{enumerate}
	\item[$(a)$]  Let $\bar S = \langle S_m : m < \bold k \rangle$, where each $S_m$ is a set (of ordinals),
	
	\item[$(b)$] $\Lambda \subseteq \bar S^{[\bar\partial]}$,
	
	\item[$(c)$] $\bar J = \langle J_m : m < \bold k \rangle$, where each $J_m$ is an ideal on $\partial_m$.
\end{enumerate}
\end{hypothesis}

To carry out our constructions in $\ZFC$, we need some combinatorial principles introduced by Shelah  \cite{Sh:227}, known as black boxes,
where he showed that they follow from ZFC (here, ZFC means the Zermelo--Fraenkel set theory with the axiom of choice). The first difficulty is to reformulate quite-free for torsion groups. To handle this, we rely extensively on techniques from algebra and set theory, with a particular focus on the use of a version of black box called the $\bar{\chi}$-black box.
To ensure the paper is self-contained, we review and extend a list of key elements from it, including:

\begin{itemize}
    \item[$\circ$] The combinatorial $\bar{\partial}$-parameter $\mathbf{x}:= (\bold k, \bar\partial, \bar S, \Lambda, \bar J) $.

    \item[$\circ$] The \emph{$\bar{\chi}$-black box} (see Definition \ref{a99}).
    \item[$\circ$] The \emph{module parameter} $\mathbf{d}$. This consists of a tuple
    \[
    \mathbf d= \langle  R, M^*, \cM, \theta      \rangle= \langle  R_{\mathbf d}, M_{\mathbf d}^*, \cM_{\mathbf d}, \theta_{\mathbf d}      \rangle
    \]
    where
     $R$ is a ring,    $M^*$ is a fixed $R$-module,
       $\cM$ is a set or class of $R$-modules, and   $\theta\geq \aleph_0$ is a regular cardinal.
    \item[$\circ$] The \emph{$\mathbf{d}$-problem} $\Xi$ which  is a set $\Xi$   of triples of the form $(G,H,h)$ consisting of $R$-modules $G$ and $H$,  and a nonzero homomorphism  $h\in \Hom_R(G,H)$.
\end{itemize}

These elements provide a solid foundation for applying Shelah's method effectively within our proofs and arguments. In Definition \ref{d691n}, we illustrate how to utilize the preceding list to construct a relatively free module with trivial duality, referred to as the $(R, \mathbf{x})$-construction $\gx$. The initial step involves employing the black box method to determine under what conditions $\mathbf{x}$ is equipped with the $\theta$-fitness. For its definition, see Definition \ref{d60n}.

Finally, we offer a solution to the trivial duality problem by extending and simplifying the existing framework \cite{Sh:1028}.
To this end, it may be worth highlighting the following technical construction. Namely, let $G_*$ be an abelian group equipped with a nonzero morphism $h \in \Hom(G_*, \mathbb{Z})$. Shelah \cite{Sh:1028} constructed an abelian group extension $G$ of $G_*$ such that $h$ cannot be extended to $\Hom(G, \mathbb{Z})$, and he implicitly  asked the following variation of
Problem \ref{p1.2}:\begin{question}\label{q14}
For  given abelian groups $G_*$, $H_*$ and  a non-zero homomorphism $h_\ast:G_*\to H_\ast$, is it possible to construct  a group extension
$G \supseteq G_*$
such that
$h_\ast$ cannot be extended to the whole group
$G$? \end{question}
Recall that
\cite[Claim 2.12]{Sh:1028}
provides a situation in which
$$\xymatrix{
	& 0 \ar[r]&Rz\ar[r]^{\subseteq}\ar[d]_{h_\ast}&G\ar[dl]^{\nexists h }\\
	&& R,
	&&&}$$where $h_\ast:Rz\to R$ is given by the assignment $z\to1\in R$ for a   distinguished  element $z$. Here,  is precisely described   the text presented as the main result of Section 3:

\begin{theorem}\label{1.3}
	Let $\mathbf x$ be a $\mathbf k$-combinatorial
	$\bar\partial$-parameter, $R=R_{\mathbf x}$ and  let  $(\chi, \bold d, \Xi)$ be a module problem. Suppose $\mathbf x$
	$\theta$-fits the triple $(\chi,\mathbf d,\Xi),\chi^+ \ge \theta + |R|^+$ and $\mathbf x$
	has $\chi$-black box. The following assertions are valid:
	\begin{enumerate}
		\item There is an $(R,\mathbf x)$-construction $\gx$ such that:
		\begin{enumerate}
			\item[(a)]    $G = G_{\gx}$ is an $R$-module of cardinality
			$|\Lambda_{\mathbf x}|$,
			
			\item[(b)]  if $(G_*,H_*,h_*) \in \Xi,$ and  $h_0 \in
			\Hom_R(G_*,G)$ is an embedding, then there is no $h_1 \in \Hom_R(G ,H_*)$ such that
			$h_0 \circ h_1 = h_*$: $$\xymatrix{
				& 0 \ar[r]&G_*\ar[r]^{h_0}\ar[d]_{h_*}&G\ar[dl]^{\nexists h_1}\\
				&& H_*
				&&&}$$
		\end{enumerate}
		\item Suppose in addition to the first item, $\mathbf x$ freely $\theta$-fits  the triple $(\chi,\mathbf d,\Xi)$.
		Then we can add the following two properties:
		
		\begin{enumerate}
			\item[(c)]  $G$ is $\sigma$-free if $\mathbf x$ is $\sigma$-free,
			
			\item[(d)] $\Hom_R(G,H_{*}) = 0$, for all $H_*$ such that $(G_*,H_*,h_*) \in \Xi$.
		\end{enumerate}
	\end{enumerate}
\end{theorem}
\sn
Our additional contributions can be summarized as follows:

\begin{enumerate}
    \item[$(a)$] It is preferable to construct $G$ with a predetermined $\Eend(G)$, but we will delay this construction until a forthcoming paper \cite{ags2}.
    \item[$(b)$] There are relatives of the condition ``$\Hom(G,\mathbb{Z})=0$'' that apply to certain classes of abelian groups where the usual notion of freeness does not apply, such as the class of  abelian $p$-groups,  denoted by $\mathbf{K}_p$, where $p$ is a prime number.
    \item[$(c)$] We also consider the restriction of the class
    $\mathbf{K}_p$ to reduced and separable objects. This class is denoted by  $\mathbf{K}_p^{[rs]}$, where
    we compute relatives of the ``trivial duality condition''.

    \item[$(d)$]   It is desirable to provide proofs in a manner that will be clear to set-theoretically minded algebraists, demonstrating how to apply these results to various algebraic questions.
\end{enumerate}

\noindent

In Definition \ref{d542n}, we present the new concept of relative freeness and almost freeness with respect to a suitable module parameter. We pay special attention to the module parameter $$\mathbf d^1_p:=\langle  R:=\bbZ, M^*:=0, \cM, \theta:= \aleph_1    \rangle,$$  where  $\cM =\{\bbZ/p^n \bbZ:n=1,2,\ldots\}$.
Also, we invent the $\mathbf d^1_p$-problem
 $\Xi^1_p,$ by looking at the class of triples $(G_*,H_*:=G_*,h_*:=\id_{G_*})$  where $G_*$ is of the form
		$$G_*:=\bigoplus\{G^*_{m, \alpha}: m < \bold{k}, \alpha < \omega_1     \},$$
		and $G^*_{m,\omega\alpha+n} \cong \mathbb{Z} / p^{n+1}\mathbb{Z}$.
The main technical task of this paper is to find
almost-free frames with respect to $\mathbf{K}_p$. The following is our second main result:

\begin{theorem}\label{1.4}
Let
$J =J_{ \aleph_1}^{\bd}$ be the ideal of bounded subsets of $\omega_1$. Then $(\aleph_1, J)$
fits the triple $(\aleph_1, \mathbf d^1_p, \Xi^1_p)$.
\end{theorem}
Section 4 is focused on proving the above central theorem,  concerning the duality of separable $p$-groups. Also, we present a
 pair $(\aleph_0, J)$ that freely fits the triple  $\mathbf d^0_p$ with $\Xi^0_p$ being the class of triples $(G,H,h)$ where:
\begin{enumerate}
\item[$-$]    $G \in \mathbf {K}_p^{[rs]}$ has cardinality $\le
2^{\aleph_0}$ and is not torsion-complete,
\item[$-$]   $H$ has the form $\bigoplus\{G_n:n \in \cU\}$ where $\cU
\subseteq \bbN$ is infinite and $G_n \cong \bbZ/p^n \mathbb{Z}$,
\item[$-$] $h$ is a non-small homomorphism from $G$ to $H$.
 	\end{enumerate}
 Namely, we present the following observation:

\begin{observation}\label{1.5}
	Let
	$J =J^{\bd}_{\aleph_0 }$ be the ideal of bounded subsets of $\omega$. Then
	$(\aleph_0, J)$ freely fits the triple $(\aleph_0, \mathbf d^0_p, \Xi^0_p)$.
\end{observation} The next part of Section 4 is closely linked to Theorem \ref{1.3}. In particular, it applies Theorem  \ref{1.4} and its relevant Observation \ref{1.5} to explore connections with small morphisms and the concept of almost relative freeness. This part provides examples to illustrate how the structural insights from Theorem \ref{1.3} manifest in these specific settings, showing how abstract set-theoretic results translate into concrete algebraic contexts:

\begin{corollary}
There is an abelian  group $G$ equipped with the following two properties:
	
	\begin{enumerate}
		\item[$(a)$]  if $(G_*, H_*, h_*) \in \Xi^0_p$,  then  every $h \in \Hom(G,H_*)$ is small.
		
		\item[$(b)$]  $G$ is $\aleph_{\omega \cdot \mathbf k}$-free with respect to $\mathbf {K}_p$.
	\end{enumerate}
\end{corollary}

Finally, we recover and extend some interesting  results from \cite{Sh:1028}, thereby answering

Problem  \ref{p1.2}:

\begin{corollary}
There is an abelian group $G$  equipped with the following two properties:

\begin{enumerate}
	\item[$(a)$]  if $G$ is   $\aleph_{\omega_1
		\cdot \mathbf k}$-free
	with respect to $\mathbf{K}_p$,
	
	\item[$(b)$]  $\Hom(G,F) = 0$ for all indecomposable $\mathbf {K}_p$-free groups $F$.
\end{enumerate}

\end{corollary}

In this paper all groups are abelian, and all rings are commutative.
For all unexplained definitions from set theoretic algebra
see
the  books by Eklof-Mekler \cite{EM02}  and G\"{o}bel-Trlifaj
\cite{GT}. Also, for  unexplained definitions from the group theory
see  Fuchs' book \cite{fuchs}.

\section{Conveniences with freeness of trees and black box}

In this section, we recall some preliminaries which are needed for the later sections of the paper. The reader may skip this section and return to it as needed.

\subsection{Freeness of relative trees}
In this subsection, we introduce a series of definitions and results with a set-theoretic emphasis. All of these will be utilized in the subsequent discussion.

\begin{notation}
Let $0<\mathbf{k}<\omega$. Suppose  $\bar\partial=\langle \partial_\ell=\partial(\ell):\ell
< \bold k\rangle$ is a sequence of regular cardinals or just limit ordinals and
$\bar S = \langle S_\ell:\ell < \bold k\rangle$ is a sequence of sets.
We allow
$\bar\partial$ to be  constant,
i.e. $ \partial_\ell = \partial$ for some $\partial$ and all $\ell < \bold k$.\end{notation}

Suppose $\cF \subseteq \prod\limits_{\ell < \bold k}{}^{S_\ell} X_\ell$ is a family of $\bold k$-sequences of functions. We say $\cF$ is {\em weakly ordinary} if for each $\bar\eta = \langle \eta_\ell : \ell < \bold k \rangle \in \cF$, each $\eta_\ell$ is a one-to-one function. In the case that the sets $S_\ell$ and $X_\ell$ are sets of ordinals, we say $\cF$ is {\em ordinary} if each $\eta_\ell$ as above is an increasing function.

\begin{definition} (\cite[Definition 0.7]{Sh:1028}).
\label{y37}
Suppose $\cF \subseteq {}^S X$ is a family of functions from $S$ into $X$, $J$ is an ideal on $S$, and $\theta$ is a cardinal.

\begin{enumerate}
    \item We say $\cF$ is {\em $(\theta, J)$-free} if for every $\cF' \subseteq \cF$ of cardinality $< \theta$, there is a sequence $\langle w_\eta : \eta \in \cF' \rangle$ such that:
    \begin{enumerate}
        \item $\eta \in \cF' \Rightarrow w_\eta \in J$, and
        \item if $\eta_1 \ne \eta_2 \in \cF'$ and $s \in S \setminus (w_{\eta_1} \cup w_{\eta_2})$, then $\eta_1(s) \ne \eta_2(s)$.
    \end{enumerate}

    \item We say $\cF$ is {\em $\theta$-free} if it is $(\theta, J)$-free where $S \subseteq \Ord$ and $J = J^{\bd}_S$, the ideal of bounded subsets of $S$.

\end{enumerate}
\end{definition}

\begin{definition} (\cite[Notation 1.2]{Sh:1028}).
	\label{a3}
Let  $\bar\partial = \langle \partial_\ell = \partial(\ell):\ell < \bold k\rangle$
and $\bar S = \langle S_\ell:\ell < \bold k\rangle$
be as described above.
\begin{enumerate}
	\item Let $\bar S^{[\bar\partial]} = \prod\limits_{\ell < \bold k}
	{}^{\partial(\ell)}S_\ell$
	and for $u \subseteq \{0,\dotsc,\bold k -1\}$ set $\bar S^{[\bar\partial,u]} = \prod\limits_{\ell \in u}
	{}^{\partial(\ell)}S_\ell$.
	Furthermore, if each $S_\ell$ is a set of ordinals, then let
	$\bar S^{< \bar\partial>} = \{\bar\eta \in
	\bar S^{[\bar\partial]}$: each $\eta_\ell$ is increasing$\}$ and
	 $\bar S^{<\bar\partial,u>}=\{\bar\eta \in
	\bar S^{[\bar\partial, u]}$: each $\eta_\ell$ is increasing$\}$.
	
	\item Suppose $\bar\eta \in \bar S^{[\bar\partial]},m < \bold k$ and $i <
	\partial_m$. Then
\begin{enumerate}
\item  for $w \subseteq \partial_m,\bar\eta
	\upharpoonleft (m,= w)$ is defined as $\langle \eta'_\ell:\ell <
	\bold k\rangle$ where
	$$
	\eta'_\ell	=\left\{\begin{array}{ll}
	\eta_\ell &\mbox{if } \ell < \bold k \wedge \ell \ne m, \\
	\eta_\ell \rest w	&\mbox{if }  \ell = m.
	\end{array}\right.
	$$

\item $\bar\eta \upharpoonleft (m,i)=\bar\eta
	\upharpoonleft (m,= \{i\})$.

\item $\bar\eta \upharpoonleft (m) =
	\langle \eta_\ell:\ell < \bold k \wedge \ell \ne m \rangle$.
\end{enumerate}	
	\item Suppose $\Lambda \subseteq \bar S^{[\bar\partial]}, m < \bold k$, $w \subseteq \partial_m$ and $u \subseteq \{0,\dotsc,\bold k-1\}$. Then
\begin{enumerate}
 \item 
  Set $\Lambda \upharpoonleft(m,=w) = \{\bar\eta
	\upharpoonleft (m,=w):\bar\eta \in \Lambda\}$.
	
	\item for $i
	\le \partial_m$ set $\Lambda \upharpoonleft(m,< i)
	= \bigcup\limits_{j<i}\Lambda \upharpoonleft (m, j)$.
	
	\item $\Lambda_{\in u} = \bigcup\{\Lambda
	\upharpoonleft(m,i):m \in u,i < \partial_m\}$. We may write ``$< m$" instead of
	``$\in m$" when ``$u = \{0,\dotsc,m-1\}$" and let
	$\Lambda_{m} = \Lambda_{\in \{m\}}$.
\end{enumerate}	
	\item We say $\Lambda \subseteq \bar S^{[\bar\partial]}$ is {\em tree-like}
	if for each $\bar\eta,\bar\nu \in \Lambda$ and $m < \bold k,$
$$\bar\eta \upharpoonleft (m,i) = \bar
	\nu \upharpoonleft (m,j) \implies \eta_m \rest i = \nu_m \rest j.$$
	
	\item We say $\Lambda \subseteq \bar S^{<\bar\partial>}$ is {\em normal} if
	whenever $\bar\eta,\bar\nu \in \Lambda, m < \bold k, i,j < \partial_m$ and
	$\eta_m(i) = \nu_m(j)$, then $i=j$.
\end{enumerate}
\end{definition}
We now recall the notion of combinatorial $\bar \partial$-parameter from \cite[Definition 1.3]{Sh:1028}:

\begin{definition}
	\label{a6}
    We say $\bold x$ is a combinatorial $\bar \partial$-parameter, when
    $\bold x = (\bold k, \bar\partial, \bar S, \Lambda, \bar J) = (\bold k_{\bold x}, \bar\partial_{\bold x}, \bar S_{\bold x}, \Lambda_{\bold x}, \bar J_{\bold x})$ and
    it satisfies:
    \begin{enumerate}
        \item[$(a)$] $\bold k \in \{1,2,\ldots\}$. Let $k = k_{\bold x} = \bold k - 1$,

        \item[$(b)$] $\bar\partial = \langle \partial_m : m < \bold k \rangle$ is a sequence of limit ordinals,

        \item[$(c)$] $\bar S = \langle S_m : m < \bold k \rangle$, where each $S_m$ is a set (of ordinals),

        \item[$(d)$] $\Lambda \subseteq \bar S^{[\bar\partial]}$,

        \item[$(e)$] $\bar J = \langle J_m : m < \bold k \rangle$, where each $J_m$ is an ideal on $\partial_m$.
    \end{enumerate}
\end{definition}

\begin{convention}
\label{a666}
Suppose that $\bold x$ is a combinatorial $\bar \partial$-parameter as above.
\begin{enumerate}
    \item If for each $\ell < \bold k$, we have $\partial_\ell = \partial$, then we may write $\partial$ instead of $\bar\partial$, and call $\bold x$ a combinatorial $(\partial, \bold k)$-parameter. This may be abbreviated as $(\partial, \bold k)$-c.p.

    \item We may say $\bold x$ is a $\bold k$-c.p. if it is an $(\aleph_0, \bold k)$-c.p.

    \item Similarly, if all $S_\ell$'s are equal to a set $S$, then we may write $S$ instead of $\bar S$.
\end{enumerate}
\end{convention}

Suppose that $\bold x$ is a combinatorial $\bar \partial$-parameter. Then
 $\bold x$  is called {\em (weakly) ordinary} if $\Lambda_{\bold x}$ is (weakly) ordinary.
	Furthermore, if
\begin{center}
$\Lambda_{\bold x} =
	\{\langle \eta_{\bold x, \ell}:\ell < \bold k_{\bold x}\rangle:$ each $\eta_{\bold x, \ell}
	\in {}^{\partial_{\bold x, \ell}}S_{\bold x,\ell}$ is increasing (one-to-one)$\}$,
\end{center}
then we call $\bold x$ {\em (weakly) ordinary full}.
Also,  $\bold x$ is {\em disjoint}, if
	$\langle S_{\bold x, \ell}:\ell < \bold k_{\bold x}\rangle$ is a sequence of pairwise
	disjoint sets.
Similarly, we say $\bold x$ is {\em free}, when $\Lambda_{\bold x} $ is free.

\subsection{The multi black box}

 We now intend to define the kind of black box that is required for our purpose. We start by defining the notion of a pre-black box.

\begin{definition} (\cite[Definition 1.7]{Sh:1028}).
	\label{a9}
	Assume $\bold x
	= (\bold k,\bar\partial,\bar S,\Lambda,\bar J)$ is a  combinatorial $\bar \partial$-parameter, and
$\bar \chi = \langle \chi_m:m < \bold k\rangle$ is a sequence of cardinals.
	\begin{enumerate}
	\item  $\bar \alpha$ is a  $(\bold x,\bar \chi)$-pre-black box, if
	\begin{enumerate}
		\item   $\bar \alpha = \langle \bar \alpha_{\bar \eta}:\bar
		\eta \in \Lambda\rangle$
		
		\item   $\bar \alpha_{\bar \eta} = \langle \alpha_{\bar \eta,m,i}:m
		< \bold k,i < \partial_m \rangle$ and
		$\alpha_{\bar\eta,m,i} < \chi_m$
		
		\item   if $\langle h_{m}: m < \bold k \rangle \in \prod\limits_{m<\bold k}{}^{\Lambda_{m}} \chi_m$, then there exists
		some $\bar \eta \in \Lambda$ such that for all
		 $m < \bold k$ and $i < \partial_m$ we have $h_m(\bar \eta
		\upharpoonleft ( m,i)) = \alpha_{\bar \eta,m,i}$.
	\end{enumerate}
We may also replace $\bold x$ by $\Lambda$ and say
	$\bar \alpha$ is a $(\Lambda,\bar \chi)$-pre-black box.
\item  We say $\bold x$ has
	$\bar\chi$-pre-black box, if some $\bar \alpha$ is a $(\bold x,\bar \chi)$-pre-black box.

\item Given $\bar\alpha$ as above, we may identify it with a function $\bold b$ with domain
	$\{(\bar\eta,m,i):\bar\eta \in \Lambda, m < \bold k,i <
	\partial_m\}$
 such that  $\bold b_{\bar\eta}(m,i) =
	\bold b(\bar\eta,m,i) = \alpha_{\bar\eta,m,i}.$
\end{enumerate}
\end{definition}

\begin{notation}
In Definition \ref{a9}, we may replace $\bar \chi$ by $\chi$, if $\bar \chi =\langle
	\chi:\ell < \bold k \rangle$, or by $\bar C = \langle
	C_\ell:\ell < \bold k\rangle$ when $|C_\ell| = \chi_\ell$ and
	$\Rang(h_\ell) \subseteq C_\ell$.
\end{notation}
\begin{definition} (\cite[Definition 1.7]{Sh:1028}).
\label{a99}
Assume $\bold x$  and
$\bar \chi$ are as in Definition \ref{a9}.
	We say $\bold x$ has a {\em $\bar \chi$-black box}, if   there exist a partition
	$\bar\Lambda = \langle \Lambda_\alpha:\alpha <
	|\Lambda|\rangle$ of $\Lambda$ and a sequence $\bold n=\langle \bar\nu_\alpha:\alpha <
	|\Lambda|\rangle$ such that:
	\begin{enumerate}
\item each $\bold x \rest \Lambda_\alpha$ has
	$\bar\chi$-pre-black box,

		\item  $\Lambda=\{\bar\nu_\alpha:\alpha < |\Lambda|\},$
		
		\item  if $\mu$ is the maximal cardinal satisfying  $(\forall \ell < \bold
		k)2^{< \mu} \le \chi_\ell$, then
 $$\alpha < \beta < \alpha + \mu
		\Rightarrow \bar\nu_\alpha = \bar\nu_\beta,$$
		
		\item  if $\alpha \le \beta < |\Lambda|$, $(\alpha, \beta) \neq (0, 0)$ and $\bar\eta
		\in \Lambda_\beta$, then $\nu_{\alpha,\bold k-1} < \eta_{\bold k-1}
		\mod J_{\bold k-1}$.
	\end{enumerate}
\end{definition}
We now recall freeness for a combinatorial parameter from \cite[Definition 1.11]{Sh:1028}.
\begin{definition}
	\label{a12}
Suppose  $\bold x$ is a    combinatorial $\bar \partial$-parameter, and $\Lambda_* \subseteq \bar S_{\bold x}^{[\bar\partial_{\bold x}]}$.
\begin{enumerate}
\item We say  $\bold x$ is {\em $\theta$-free over $\Lambda_*$},
	if it is weakly ordinary and
	for every $\Lambda \subseteq \Lambda_{\bold x} \backslash \Lambda_*$
	of cardinality $< \theta$, there is a list $\langle \bar\eta_\alpha:\alpha <
	\alpha_*\rangle$ of $\Lambda$ such that for every $\alpha$, for some
	$m < \bold k_{\bold x}$ and $w \in J_{\bold x,m}$, if $$\bar\nu \in \{\bar\eta_\beta:\beta
	< \alpha\} \cup \Lambda_* \quad and \quad\bar\nu \upharpoonleft (m) = \bar\eta_\alpha
	\upharpoonleft(m),$$then we can deduce that $\nu_m(j) \ne \eta_{\alpha,m}(i)  $ for all $ j < \partial_{\bold x,m} $  and $i \in \partial_{\bold x,m} \setminus w$. If $\Lambda_{\bold x}$ is normal, we can restrict ourselves to $i = j$ and this is the usual case.
\item Suppose $\bar\Lambda = \langle \Lambda_{\bar\nu}:\bar\nu
	\in \Lambda_{\bold x}\rangle$ where each $\Lambda_{\bar\nu} \subseteq
	\Lambda_{\bold x}$. We say $\bold x$ is {\em $\theta$-free over $\Lambda_*$ respecting
	$\bar\Lambda$} if for every $\Lambda \subseteq \Lambda_{\bold x}
	\setminus \Lambda_*$ of cardinality $< \theta$, there is a list
	$\langle \bar\eta_\alpha:\alpha < \alpha_*\rangle$ of $\Lambda$ witnessing
	$\bold x$ is $\theta$-free over $\Lambda_*$ such that for every $\alpha < \alpha_*$,
	 $$\bar\eta_\alpha \in \Lambda_{\bar\nu}\implies \bar\nu \in \{\bar\eta_\beta:\beta <
		\alpha\} \cup \Lambda_*.$$	
\end{enumerate}
\end{definition}

\begin{discussion}\label{exmbb}
 The existence problem of
	$\bar\chi$-black boxes, equipped with
	the above freeness properties, is the subject of
	\cite[1.20 and 1.25]{Sh:1028}.
\end{discussion}

\section {The relative notions of freeness and module parameters} \label{1}
The main result of this section is Theorem \ref{d64n}. In \cite{Sh:1028}, Shelah constructs abelian groups and modules which are, on the one hand, quite free and, on the other hand, have a small dual. The results in \cite{Sh:1028} do not apply directly  to the classes $\mathbf{K}_p$ and $\mathbf{K}_p^{[rs]}$.

\begin{notation} If $G$ is an abelian group and $n\in\mathbb{N}$, then set:
	\begin{enumerate}
		\item $nG:=\{ng:g\in G\}$,
			\item $G[n]:=\{g\in G:ng=0\}$,
		\item $\ord(g)$ means the order of an element $g$, \item $\textbf{ht}_p(g)$ stands for transfinite   height of the element $g$ at prime $p$.
			\end{enumerate}
	\end{notation}

\begin{definition}
	\label{d52na}
	Let $p$ be a prime number.
By a  $p$ -group is meant a group the orders of whose elements are powers of   $p$.
 Recall that  $p$-groups without elements of infinite
heights  are called separable. A reduced group means a group with no nonzero divisible subgroup.
\end{definition}

\begin{definition}
\label{d52n}
Let $p$ be a prime number.
\begin{enumerate}
\item Let $\mathbf {K}_p$ be the class of abelian $p$-groups.
Also, let $\mathbf {K}_p^{[rs]}$ be the class of abelian $p$-groups $G$ which are
reduced and separable.
\item Suppose $G,H \in \mathbf {K}_p$. A map $g \in \Hom(H,G)$ is called {\em small}, if  the Pierce condition
$p^n H[p^k]\subseteq
\Ker(g)$ holds, with the convention that $p^n H[p^k]=p^n H\cap H[p^k]$.
In  means that for every $k > 0$,
there exists $n > 0$ such that
$\ord(a)\leq p^k$
 and $\textbf{ht}_p(a)\geq n$ imply that
$g(a)=0$.\end{enumerate}
\end{definition}
\begin{definition}\begin{enumerate}
\item An abelian group $G$ is called {\em $\mathbf {K}_p$-free}, provided it is the direct sum of finite
   cyclic $p$-groups.

\item An abelian group $G$ is called {\em $(\theta,\mathbf {K}_p)$-free}, if  every $H
\subseteq G$ of cardinality $< \theta$ is $\mathbf {K}_p$-free.
\end{enumerate}
\end{definition}

We now give, in a series of definitions, a more general notion, that we will work with.
\begin{definition}
\label{d54n}
A {\em module parameter} is a tuple
\[
\mathbf d= \langle  R, M^*, \cM, \theta      \rangle= \langle  R_{\mathbf d}, M_{\mathbf d}^*, \cM_{\mathbf d}, \theta_{\mathbf d}      \rangle
\]
where:
\begin{enumerate}
\item[$(a)$]   $R$ is a ring,

\item[$(b)$]   $M^*$ is a fixed $R$-module,

\item[$(c)$]   $\cM$ is a set or class of $R$-modules,

\item[$(d)$]   $\theta\geq \aleph_0$ is a regular cardinal.
\end{enumerate}
\end{definition}
Given a module parameter $\mathbf d$, we define some classes of $R_{\mathbf d}$-modules as follows.
\begin{definition}
\label{d541n}
Suppose $\mathbf d=\langle  R_{\mathbf d}, M_{\mathbf d}^*, \cM_{\mathbf d}, \theta_{\mathbf d}      \rangle$ is a module parameter.
\begin{enumerate}
\item Let $\mathbf K_{\mathbf d}$ be the class of $R_{\mathbf d}$-modules.

\item Let $\mathbf K^{\fr}_{\mathbf d}$ be the class of  $R_{\mathbf d}$-modules $G$ which are
{\em $\mathbf d$-free}, i.e.,
 $G = \bigoplus\{M_s:s \in I\} \oplus M$
 where $M \cong
M^*_{\mathbf d}$ and each $M_s$ is isomorphic to some member of
$\cM_{\mathbf d}$; here $\fr$ stands for free.

\item Let $\mathbf K^{\sfr}_{\mathbf d}$ be the class of  $R_{\mathbf d}$-modules $G$ which are
{\em semi-$\mathbf d$-free}, i.e.,  $G = \bigoplus\{M_s:s \in I\}$  where  each $M_s$ is isomorphic to some member of
$\cM_{\mathbf d}$; here $\sfr$ stands for semi-free.
\end{enumerate}
\end{definition}

We also define the notion of freeness of one $R$-module over another  with respect to a module parameter.
\begin{definition}
\label{d542n}
Suppose $\mathbf d$ is a module parameter. For $R_{\mathbf d}$-modules $M_1,M_2$, we say that
 {\em $M_2$ is $\mathbf d$-free over $M_1$}, when $M_1
  \subseteq M_2$ are from $\mathbf K_{\mathbf d}$ and for some $N \in
  \mathbf K_{\mathbf d}^{\sfr}$ we have $M_2 =  M_1 \oplus N$. In the case  $\mathbf d$ is clear from the context, we may say   $M_2$ is free over $M_1$.
\end{definition}

Note that an $R_{\mathbf d}$-module  is $\mathbf d$-free if and only if
 it is $\mathbf d$-free over some $M\cong M^*_{\mathbf d}$.

\begin{definition}
\label{d543n}
Suppose $\mathbf d= \langle  R_{\mathbf d}, M_{\mathbf d}^*, \cM_{\mathbf d}, \theta_{\mathbf d}      \rangle$ is a module parameter and $\theta$ is an infinite cardinal.
\begin{enumerate}
\item Let $\mathbf K^{\fr}_\theta = \mathbf K^{\fr}_{\mathbf d,\theta}$ be
the class of $R_{\mathbf d}$-modules $M$ which are {\em $(\mathbf d,\theta)$-free}, this means that there are $\bar M, I$ such that:
\begin{enumerate}
\item[$(a)$]  $I$ is a $\theta$-directed partial order,

\item[$(b)$]  $\bar M$ is a sequence $\langle M_s:s \in I\rangle$
of members of $\mathbf K^{\fr}_{\mathbf d}$,

\item[$(c)$]  $I$ has a minimal member $\min(I)$ such that
  $M_{\min(I)} \cong M^*_{\mathbf d}$,

\item[$(d)$]   $s <_I t \Rightarrow M_t$ is free over $M_s$,

\item[$(e)$]   $M = \bigcup\{M_s:s \in I\}$,

\item[(f)]  each $M_s$ has cardinality $< \theta$.
\end{enumerate}If $\mathbf d$ is clear from the context, we may say   $M$ is $  \theta$-free.
\item We say {\em $M_2$ is $(\mathbf d, \theta)$-free
 over $M_1$} if
 $M_1$ is a sub-module of $M_2$ and there are $\bar M, I$ as in clause
(1) with $M_{\min(I)} = M_1$ and $M=M_2.$ If $\mathbf d$ is clear from the context, we say   $M_2$ is $  \theta$-free
	over $M_1$.
\end{enumerate}
\end{definition}

We also define another variant of the above classes of modules.
\begin{definition}
\label{d56n}
 Suppose $\mathbf d$ is a module parameter. The class $\mathbf K^{\text g}_{\mathbf d}$ is defined as
the class of all  $M \in \mathbf K_{\mathbf d}$ extended  by the individual constants $c^M$ for
$c \in M_{\mathbf d}^*$ such that $c \mapsto c^M$ is
an embedding of $M_{\mathbf d}^*$ into $M$.
The classes
$\mathbf K^{\gr}_{\mathbf d}, \mathbf K^{\gr}_{\mathbf d,\theta}$ and $\mathbf K^{\sgr}_{\mathbf d}$  are defined in a similar way using the classes  $\mathbf K^{\fr}_{\mathbf d}, \mathbf K^{\fr}_{\mathbf d,\theta}$ and
$\mathbf K^{\sfr}_{\mathbf d}$  respectively. So, $\mathbf K^{\sgr}_{\mathbf d}=\{(M,c_a)_{a\in M^*_{\mathbf d}}:M\in\mathbf K^{\sfr}_{\mathbf d}\emph{  with } c_a=0\}$.
\end{definition}
In the sequel, we aim to generalize \cite[Definition 2.11]{Sh:1028} of $\theta$-fitness to the more general context of module parameters.
\begin{definition}
\label{d57n}
Suppose $\mathbf d$ is a  module-parameter.  A {\em $\mathbf d$-problem}  is a set $\Xi$   of triples of the form $(G,H,h)$ satisfying:
\begin{enumerate}

\item[$(\alpha)$]  $G$ and $H$ are $R_{\mathbf d}$-modules,

\item[$(\beta)$]  $h$ is a nonzero homomorphism from $G$ to $H$ as an $R_{\mathbf d}$-module homomorphism.
\end{enumerate}
\end{definition}
\begin{definition}
\label{d61n}
 We say $(\chi,\mathbf d,\Xi)$ is a {\em module problem}, when
\begin{enumerate}
\item $\chi$ is an infinite cardinal,

\item  $\mathbf d:= \langle  R_{\mathbf d}, M_{\mathbf d}^*, \cM_{\mathbf d}, \theta_{\mathbf d}      \rangle$ is a module parameter,

\item  $\Xi$ is a $\mathbf d$-problem,


\item   if 
$(G_*,H_*,h_*) \in \Xi$, then
$|H_*| + |G_*| \le \chi$,

\item  $\Xi$ has cardinality $\le \chi$,

\item  $\cM_{\mathbf d}$ and each $M \in \cM_{\mathbf d}$ have cardinality $\le \chi$,

\item  $M^*_{\mathbf d}$ has cardinality $\le \chi$.
\end{enumerate}
\end{definition}
The following gives the promised generalization of \cite[Definition 2.11]{Sh:1028}.
\begin{definition}
\label{d60n}
Suppose $\chi$ is an infinite cardinal,  $\mathbf d$ is a module-parameter and $\Xi$ is a $\mathbf d$-problem such that $(\chi, \mathbf d, \Xi)$ is a module problem and set $\theta=\theta_{\mathbf{d}}$. Also, assume $\mathbf x$ is    a combinatorial
	$\bar\partial$-parameter.
  \begin{enumerate}
\item  We say $(\bar\partial,\bar{J})$  $\theta$-fits  the triple
$(\chi,\mathbf d,\Xi)$, when the following conditions $(A)$ and $(B)$ are satisfied, where $R=R_{\mathbf d}$:
\begin{enumerate}
\item[(A)]
\begin{enumerate}
\item[(a)]  $\bar\partial = \langle \partial_m : m < \bold k \rangle$ is a sequence of limit ordinals,
and $\bar J = \langle J_m : m < \bold k \rangle$, where each $J_m$ is an ideal on $\partial_m$,

\item[(b)] $(\chi,\mathbf d,\Xi)$ is a module problem,

\end{enumerate}

\item[(B)]  Suppose that
\begin{itemize}
\item[(a)] $(G_*,H_*,h_*) \in \Xi$,
\item[(b)] $M_* \in
K^{\fr}_{\mathbf d}$ and $M_*=  M^* \oplus N$ for some $N\in K^{\sfr}_{\mathbf d}$,

\item[(c)] $h_0 \in \Hom_R(G_*,M_*)$,
\item[(d)] $\bar M = \langle M_{m,i}:m < \mathbf k,
i < \partial_m\rangle$ is such that $M_{m,i} \in K^{\sfr}_{\mathbf d}$ for
$m < \mathbf k,i < \partial_{m}$ and $\sum\limits_{m,i} \|M_{m,i}\| < \theta$,
\item[(e)] $G_0 = \bigoplus\{M_{m,i}:m < \mathbf k,i <\partial_m\} \oplus M_*$,
so $G_0 \in \mathbf
K^{\fr}_{\mathbf d}$,

\item[(f)] $h_1 \in \Hom_R(G_0,H_*)$
is such that $h_0 \circ h_1:G_* \rightarrow
H_*$ is equal to $h_*$, i.e., the following diagram commutes: $$\xymatrix{
	&&G_\ast\ar[r]^{h_0}\ar[dr]_{h_\ast}& M_\ast \ar[r]^{\subseteq}& G_0\ar[ld]^{h_1}&\\
	&&& H_\ast&
	&&&}$$
\end{itemize}

Then there is $G_1$ such that:

\begin{enumerate}
\item[$(\alpha)$]  $G_1$ is an $R$-module extending $G_0$,

\item[$(\beta)$]  $G_1$ has cardinality $< \theta$,

\item[$(\gamma)$]   there is no $R$-homomorphism  $f$ from
  $G_1$ into $H_*$ extending $h_1$.
\end{enumerate}
 We say $(\partial, J)$ $\theta$-fits the triple
$(\chi,\mathbf d,\Xi)$, when the sequences $\bar\partial,\bar{J}$  are fixed.
\end{enumerate}The definition is neatly summarized in the accompanying diagram:
$$\xymatrix{
&&G_\ast\ar[r]^{h_0}\ar[dr]_{h_\ast}& M_\ast \ar[r]^{\subseteq}& G_0 \ar[r]^{\subseteq}\ar[d]_{h_1}&G_1\ar[dl]^{\nexists f}\\
&&& H_\ast\ar[r]^{=}& H_{\ast}
&&&}$$
\item We say {\em $(\bar\partial,\bar{J})$ freely $\theta$-fits the triple
$(\chi,\mathbf d,\Xi)$}, if in addition it satisfies:

\begin{enumerate}
\item[$(\delta)$]  if $m < \mathbf k$ and $u \in
J_{m}$, then $G_1$ is $\mathbf d$-free over  $$\bigoplus_{\ell < \mathbf k}\{M_{\ell,i}: i < \partial_\ell\emph{ and }\ell = m
\Rightarrow i \in u\} \oplus M_*.$$

\end{enumerate}
\item If $\mathbf x$ has a $\chi$-black box, then we say $\mathbf x$ (freely) $\theta$-fits the triple
$(\chi,\mathbf d,\Xi)$, when $(\bar\partial_{\mathbf x},\bar J_{\mathbf x})$ (freely)  $\theta$-fits the triple
$(\chi,\mathbf d,\Xi)$.
\item In the above definitions, we may omit $\theta$ when $\theta = |R|^+ +
\max\{\partial^+_{\mathbf x,m}:m < \mathbf k_{\mathbf x}\}$. Also, we may say the paring fits $H_\ast$.
\end{enumerate}
\end{definition}

\begin{definition}
\label{d62n}
Suppose $\bold x$ is a combinatorial $\bar\partial$-parameter, the triple $(\chi, \bold d, \Xi)$ is a module problem
and $R=R_{\bold d}$.
\begin{enumerate}
\item An {\em $R$-module $G$ is $(\chi, \bold d, \Xi)$-derived from  $\mathbf x$}, when there is a tuple
   \[
   \langle X_*, G_*, \langle G_{\bar\eta}:\bar\eta
\in \Lambda_{\mathbf x} \rangle, \langle Z_{\bar\eta}:\bar\eta
\in \Lambda_{\mathbf x} \rangle            \rangle
   \]
   such that:

\begin{enumerate}
\item[(a)]
$X_* = \{M_{\bar\eta \upharpoonleft (m,i)}:m < \mathbf k_{\mathbf x},i <
\partial_m$ and $\bar\eta \in \Lambda_{\mathbf x}\} \cup \{M^*_{\bold d}  \},$ where each
$M_{\bar\eta \upharpoonleft (m,i)}$ is semi-$\bold d$-free,
\item[(b)] $G_*=\bigoplus\limits_{M \in X_*}M$,

\item[(c)]   the $R$-module $G$ is generated by $\bigcup\{G_{\bar\eta}:\bar\eta
\in \Lambda_{\mathbf x}\} \cup G_*$, so $G_* \subseteq G$,

\item[(d)]  $G/G_*$ is the direct sum of $\langle (G_{\bar\eta} +
G_*)/G_*:\bar\eta \in \Lambda_{\mathbf x}\rangle$,

\item[(e)]  $Z_{\bar\eta} \subseteq X_*$
for $\bar\eta \in \Lambda_{\mathbf x}$,

\item[(f)]   if $\bar\eta \in \Lambda_{\mathbf x}$, then  the
$R$-submodule $G_{\bar\eta} \cap G_*$ of $G$ is equal to the
$R$-submodule generated by
$$\bigcup\{M_{\bar\eta \upharpoonleft (m,i)}:m < \mathbf k_{\mathbf x}\emph{ and
}i < \partial_{\mathbf x,m}\} \cup \bigcup \{M: M \in         Z_{\bar\eta}  \}.$$
\end{enumerate}
\item By an {\em $(R,\mathbf x, \chi, \bold d, \Xi)$-construction}, we mean a tuple
 $$\gx=\langle \mathbf x, R, G_*, G, \langle M_{\bar\eta}:\bar\eta \in
  \Lambda_{\mathbf x, < \mathbf k_{\mathbf x}}\rangle, \langle
  G_{\bar\eta},Z_{\bar\eta}:\bar\eta \in \Lambda_{\mathbf x}\rangle \rangle$$
  defined as above, and we may say $\gx$ is $(\chi, \bold d, \Xi)$-derived from $\mathbf x$.
  \end{enumerate}
\end{definition}
\begin{notation}
Given an $(R, \mathbf x, \chi, \bold d, \Xi)$-construction $\gx$,   we shall  write $G^{\gx}_*$ for $G_*$, $G_{\gx}$ for $G$, $G_{\gx,\bar\eta}$ for
$G_{\bar\eta}$, and etc.
Furthermore, we may
remove the index $\gx$,
when it is clear from the context.
\end{notation}
The following definition presents several variants of the above concept and can be viewed as a generalization of \cite[Definition 2.4]{Sh:1028}.
\begin{definition}
\label{d691n}
Suppose $\mathbf x$, $(\chi, \bold d, \Xi)$, $G$ and $\gx$ are as in Definition \ref{d62n}.
\begin{enumerate}
\item We say $G$ is {\em simple}   when $Z_{\bar\eta} = \{M^*_{\bold d}\}$ for
every $\bar\eta \in \Lambda_{\mathbf x}$.

\item We say $\gx$ is {\em almost simple}, when for each $\bar\eta \in
\Lambda_{\mathbf x},$ we have $|Z_{\bar\eta} \backslash \{M^*_{\bold d}\}| \le 1$.

\item We say $G$  is {\em $(\chi, \bold d, \Xi)$-freely derived from $\mathbf x$},
if in addition to items (a)-(f) of Definition \ref{d62n}(1) we have
\begin{enumerate}
\item[$(g)$]  if $\bar\eta \in \Lambda_{\mathbf x},m < \mathbf k$ and $w
  \in J_{\mathbf x,m}$,  then there exists some $R$-module $G^\perp_{\bar\eta,m,\omega}$ such that   $G_{\bar\eta} = G_{\bar\eta,m,w} \oplus
  G^\perp_{\bar\eta,m,w}$ is a free $R$-module where
$G_{\bar\eta,m,w}$ is
the $R$-submodule of $G$ generated by
$$\bigcup\{M_{\bar\eta \upharpoonleft (m_1,i_1)}:m_1 < \bold k_{\bold x}, i_1 < \partial_{m_1}, \emph{ and }(m_1 = m \Rightarrow i_1 \in w)\} \cup \bigcup\{M: M \in Z_{\bar\eta}\}.$$
\end{enumerate}

\item We say $\gx$ is a {\em canonical $(R, \bold x, \chi, \bold d, \Xi)$-construction},  if we have
$Z_{\bar\eta} = 0$      for all $\bar\eta \in \Lambda_{\gx}$.

\item  We say $\gx$ is {\em $(< \theta)$-locally free}, if in addition,   the following two properties are valid:
\begin{enumerate}
\item[$(h)$]  it satisfies clause (3)(g), furthermore the quotient
  $G^\perp_{\bar\eta,n,w}$ is $\theta$-free,

\item[$(i)$]  $\mathbf x$ is $\theta$-free.
\end{enumerate}
\item
 $\gx$   is called well
	orderable, when we can find $\bar\Lambda$ so that:

\begin{enumerate}
	\item[$(\alpha)$]  $\bar\Lambda = \langle \Lambda_\alpha:\alpha \le
	\alpha_*\rangle$ is $\subseteq$-increasing and continuous,
	
	\item[$(\beta)$]  $\Lambda_{\alpha_*} = \Lambda_{\mathbf x}$
	and $\Lambda_0 = \emptyset$,
	
	\item[$(\gamma)$]   if $\bar\eta \in \Lambda_{\alpha +1} \setminus
	\Lambda_\alpha$ and $m < \mathbf k_{\bold x}$ then
	$$\big\{i < \partial_m:(\exists
	\bar\nu \in \Lambda_\alpha)(\bar\eta \upharpoonleft (m,i) = \bar\nu
	\upharpoonleft (m,i)\big\} \in J_{\mathbf x,m},$$
	
	\item[$(\delta)$]  if $\bar\eta \in \Lambda_{\alpha +1} \setminus
	\Lambda_\alpha$ then $\bigcup\{M: M \in Z_{\bar\eta}\} \subseteq \langle
	\bigcup\{G_{\bar\nu}:\bar\nu \in \Lambda_\alpha\}\rangle_G$.
\end{enumerate}\end{enumerate}
\end{definition}

Now, we are ready to confirm
{Question} \ref{q14}:

\begin{theorem}
	\label{d64n} Assume $0<\mathbf{k}<\omega$ and $\bar\partial$ of length $\mathbf k$ are given.
Let $\mathbf x$ be a $\mathbf k$-combinatorial
	$\bar\partial$-parameter, $R=R_{\mathbf x}$ and  let  $(\chi, \bold d, \Xi)$ be a module problem. Suppose $\mathbf x$
	$\theta$-fits the triple $(\chi,\mathbf d,\Xi),\chi^+ \ge \theta + |R|^+$ and $\mathbf x$
	has $\chi$-black box. The following assertions are valid:
	\begin{enumerate}
		\item There is $\gx$ such that:
		\begin{enumerate}
			\item[(a)]  $\gx$ is an $(R,\mathbf x)$-construction,
			\item[(b)]  $G = G_{\gx}$ is an $R$-module of cardinality
			$|\Lambda_{\mathbf x}|$,
			
			\item[(c)]  if $(G_*,H_*,h_*) \in \Xi,$ and  $h_0 \in
			\Hom_R(G_*,G)$ is an embedding, then there is no $h_1 \in \Hom_R(G ,H_*)$ such that
			$h_0 \circ h_1 = h_*$: $$\xymatrix{
				& 0 \ar[r]&G_*\ar[r]^{h_0}\ar[d]_{h_*}&G\ar[dl]^{\nexists h_1}\\
				&& H_*
				&&&}$$
			
			\item[(d)]  $\gx$ is simple.
		\end{enumerate}
		
		\item Suppose in addition to the first item, $\mathbf x$ freely $\theta$-fits  the triple $(\chi,\mathbf d,\Xi)$.
		Then we can add the following two properties:
		
		\begin{enumerate}
			\item[(e)]  $G$ is $\sigma$-free if $\mathbf x$ is $\sigma$-free,
			
			\item[(f)] $\Hom_R(G,H_{*}) = 0$, for all $H_*$ such that $(G_*,H_*,h_*) \in \Xi$.
		\end{enumerate}
	\end{enumerate}
\end{theorem}

\begin{proof}
(1) We are going to define
$$\gx=\langle R,X_*, G_*, G, \langle M_{\bar\eta}:\bar\eta \in
	\Lambda_{\mathbf x, X_*,< \mathbf k_{\mathbf x}}\rangle, \langle
	G_{\bar\eta},Z_{\bar\eta}:\bar\eta \in \Lambda_{\mathbf x}\rangle \rangle,$$equipped with the requested property.
Recall that the module parameter $\mathbf d$ is of the form $\langle  R_{\mathbf d}, M_{\mathbf d}^*, \cM_{\mathbf d}, \theta_{\mathbf d}      \rangle$.
For
every $\bar\eta \in \Lambda_{\mathbf x}$, we set
$Z_{\bar\eta} := \{M^*_{\bold d}\}$ and $R:=R_{\mathbf x}$.
Let $\langle M_{\mathbf d, \alpha} :\alpha < \alpha_{\mathbf d} \le \chi\rangle$ enumerate $\cM_{\mathbf d}$. For $\bar\eta \in \Lambda_{\mathbf x}$, $m< \mathbf{k}_{\mathbf x}$, and $i< \partial_m$, set
	\[
	M_{\bar\eta \upharpoonright (m,i)}:=\bigoplus\{M_{\bar\eta \upharpoonright (m,i), \alpha}: \alpha < \alpha_{\mathbf d} \},
	\]
	where for each $\alpha$ as above, $M_{\bar\eta \upharpoonright (m,i),\alpha} \cong M_{\mathbf d,\alpha}$. Let also $f_{\bar\eta  \upharpoonright (m,i),\alpha}$ be an isomorphism from $M_{\mathbf d,\alpha}$ onto $M_{\bar\eta \upharpoonright (m,i),\alpha}$. We next define:
	
	\begin{itemize}
		\item[] $X_* = \{M_{\bar\eta \upharpoonright (m,i)}:m < \mathbf k_{\mathbf x}, i < \partial_m, \bar\eta \in \Lambda_{\mathbf x}\} \cup \{M^*_{\mathbf d} \},$
		\item [] $G_* = \bigoplus\{M_{\bar\eta,\alpha}:\alpha < \alpha_{\mathbf d}, \bar\eta \in \Lambda_{\mathbf x,< \mathbf k}\} \oplus M^*_{\mathbf d}.$
	\end{itemize}
Let the collection $\{(\alpha_\varepsilon,G_\varepsilon, H_\varepsilon,h_\varepsilon, g_\varepsilon, f_\varepsilon): \varepsilon < \chi\}$ list, possibly with repetitions, all tuples $(\alpha,\mathbb{G},\mathbb{H},h,g, f)$ satisfying:
\begin{itemize}
		\item[$(i)$] $\alpha < \alpha_{\mathbf d},$
		\item[$(ii)$] $(\mathbb{G},\mathbb{H},h) \in \Xi,$
		\item[$(iii)$] $g \in \text{Hom}_R(M_{\mathbf d,\alpha}, \mathbb{H}),$
		\item[$(iv)$] $f \in \text{Hom}_R(\mathbb{G}, M^*_{\mathbf d}).$
	\end{itemize}
Let $\mathbf b$ be a $\chi$-black box for $\mathbf x$ (see Definition \ref{a9}). There are  $\mathbf b'$ and $\mathbf b''$ such that:
	\[
	\mathbf b_{\bar\eta}(m,i) = \pr(\mathbf b'_{\eta}(m,i), \mathbf b''_{\bar\eta}(m,i)),
	\]
where $\pr(-,-)$ denotes a pairing function.
For $\bar\eta \in \Lambda_{\mathbf x}$ and $\alpha < \alpha_{\mathbf d}$, set
\[	G^0_{\bar\eta, \alpha} = \sum\{M_{\bar\eta \upharpoonright (m,i), \alpha}:  m < \mathbf k, i < \partial_m\} \oplus M^*_{\mathbf d} \subseteq G_*.
	\]	
Now, suppose $\varepsilon < \chi$. Define $F_{\bar\eta, \varepsilon}: G^0_{\bar\eta, \alpha_\varepsilon} \to H_\varepsilon$ such that for each $m < \mathbf k$ and $i < \partial_m$,
	\[
	y \in M_{\mathbf d,\alpha_\varepsilon} \Rightarrow h_{\bar\eta}(F_{\bar\eta  \upharpoonright (m,i),\alpha_\varepsilon}(y)) = g_{\varepsilon}(y) \in H_\varepsilon.
	\]
	
We apply the property supported by Definition \ref{d60n} to the data $(G_\varepsilon, H_\varepsilon, h_\varepsilon)$, together with $f_\varepsilon$, $G^0_{\bar\eta, \alpha_\varepsilon}$, and $F_\varepsilon$. This allows us to conclude that, assuming $F_\varepsilon \circ f_\varepsilon = h_\varepsilon$, there exists some $G^1_{\bar\eta, \alpha_\varepsilon}$ extending $G^0_{\bar\eta, \alpha_\varepsilon}$ such that no homomorphism from $G^1_{\bar\eta, \alpha_\varepsilon}$ into $H_\varepsilon$ extends $F_\varepsilon$:
$$\xymatrix{
		&&G_\varepsilon\ar[r]^{f_\varepsilon}\ar[dr]_{h_\varepsilon}& M^*_{\bold d} \ar[r]^{\subseteq}& G^0_{\bar\eta, \alpha_\varepsilon} \ar[r]^{\subseteq}\ar[d]_{F_{\bar\eta, \varepsilon}}&G^1_{\bar\eta, \alpha_\varepsilon}\ar[dl]^{\nexists }\\
		&&& H_\varepsilon\ar[r]^{=}& H_{\varepsilon}
		&&&}$$
	
Finally, let $G_{\bar\eta}$ be freely generated  by $\bigcup\limits_{\varepsilon < \chi}G^1_{\bar\eta, \alpha_\varepsilon}$ excepted with relations coming from them. These define the $(R,\mathbf x)$-construction $\gx$. It remains to show that it is as required. Suppose not. Then we can find some $(G_*, H_*, h_*) \in \Xi$ and an embedding $h_0 \in \text{Hom}_R(G_*, G)$ such that there exists $h_1: G \to H_*$ satisfying $h_1 \circ h_0=h_*$:$$\xymatrix{
		& 0 \ar[r]&G_*\ar[r]^{h_0}\ar[d]_{h_*}&G\ar[dl]^{\exists h_1}\\
		&& H_*
		&&&}$$
Thanks to the black box, we can find $\bar\eta$, $m$, and $i$ such that if we set $\varepsilon=\mathbf b_{\bar\eta(m, i)}$, then
	\[
	(G_*, H_*, h_*, g, f):=(G_\varepsilon, H_\varepsilon, h_\varepsilon, g_\varepsilon, f_\varepsilon),
	\]
	where $g_\varepsilon, f_\varepsilon$ are chosen so that $h_1\rest \circ f_{\bar\eta  \upharpoonright (m,i),\alpha_\varepsilon}=g_\varepsilon$ and $f_\varepsilon=\pi \circ h_0$ with the convention that $\pi: G \to M^*_{\mathbf d}$ is the projection map. Let us summarize these with diagrams:$$\xymatrix{
&M_{\mathbf d,\alpha_\varepsilon}\ar[rr]^{f_{\bar\eta  \upharpoonright (m,i),\alpha_\varepsilon}}\ar[d]_{g_\varepsilon}&&M_{\bar\eta \upharpoonright (m,i),\alpha}\ar[d]_{\subseteq}\\
			& H_*&&G\ar[ll]_{h_1}
		}\xymatrix{
			&G_\varepsilon\ar[r]^{=}\ar[d]_{g_\varepsilon}&  G_*\ar[d]^{h_0}\\
			& M_{\mathbf d,\alpha_\varepsilon}&G\ar[l]_{\pi}
			&&&}$$
Therefore, $h_1 \restriction G^1_{\bar\eta, \alpha_\varepsilon}$ extends $F_{\bar\eta, \varepsilon}$:$$\xymatrix{
	&&&G^0_{\bar\eta, \alpha_\varepsilon}\ar[r]^{F_{\bar\eta, \varepsilon}}\ar[d]_{\subseteq}&  H_\ast\ar[d]^{=}\\
		&&& G^1_{\bar\eta, \alpha_\varepsilon}\ar[r]^{h_1 \restriction }&H_\ast
		&&&}$$

This is in 	 contradiction with the definition of $G^1_{\bar\eta, \alpha_\varepsilon}$.

	(2) Clause (e) is essentially \cite[Claim 2.12]{Sh:1028}, so
 we only show $(f)$. Recall $\theta = \cf(\theta)$ is $> \|G_*\|$ whenever
	$(G_*,H_*,h_*) \in \Xi$.  By induction on $i \leq \theta$, we choose an increasing and continuous chain $\langle M_i: i \leq \theta \rangle$
	of elements of $\mathbf K_{\mathbf d}$ as follows:
	\begin{itemize}
		\item[$(i)$] For $i=0$, let $M_0$ be the expansion of $M^*_{\mathbf d}$ by $c^M=c$ for
		$c \in M^*$.
		
		\item[$(ii)$] For $i$ a limit ordinal, let $M_i = \bigcup\{M_j:j <i\}$.
		
		\item[$(iii)$] For $i=j+1$, let  $\big\langle
		(G_{j,\alpha},H_{j,\alpha},h_{j,\alpha},g_{j,\alpha},M_{j,\alpha}):\alpha
		< \alpha_j\big\rangle$  be an enumeration of the set
		$\cM_i$ consisting of all tuples
		$(G_*,H_*,h_*,g,M)$
		where
		 	\begin{itemize}
		 	\item  $M \subseteq_R
			M_j$,
			 	\item $g:M\stackrel{\cong}\longrightarrow G_*$,
			 	\item $(G_*,H_*,h_*) \in
			\Xi.$	\end{itemize}
For each $\alpha < \alpha_j$, let  $N_{j,\alpha} \in \mathbf K_{\mathbf d}$ be such that it
		extends $M_{j,\alpha}$  such that $N_{j,\alpha}$ is $\theta$-free over
		$M_{j,\alpha}$ in $\mathbf K_{\mathbf d}$, and there is no $h \in
		\Hom_R(N_{j,\alpha},H_{j,\alpha})$ extending $h_{j,\alpha} \circ
		g_{j,\alpha}$. We have the following commutative diagram:
		$$\xymatrix{
			&& M_{j,\alpha} \ar[d]_{\subseteq}\ar[r]^{g_{j,\alpha}}& G_{j,\alpha}\ar[d]^{h_{j,\alpha}}&\\
			&& N_{j,\alpha}\ar[r]^{\nexists h}& H_{j,\alpha}
			&&&}$$
	\Wilog \, $N_{j,\alpha} \cap M_j = M_{j,\alpha}$ and $\langle
		N_{j,\alpha} \backslash M_{j,\alpha}:\alpha < \alpha_j\rangle$ are
		pairwise disjoint. Let $M_i$ be the $R$-module generated by
	 $\bigcup\{N_{j,\alpha}:\alpha < \alpha_j\},$  freely except that it
		extends $M_j$ and $N_{j,\alpha}$ for $\alpha < \alpha_j$.\end{itemize}
We will show that \( G := M_\theta \) satisfies the required properties. To see this, let \( (G_*, H_*, h_*) \in \Xi \), and suppose, for the sake of contradiction, that there exists a nonzero homomorphism \( h: G \to H_* \).
Due to the construction, we can find some \( i < \theta \) such that \( h \restriction M_{i+1} \) is nonzero. This implies that there exists some \( \alpha < \alpha_j \) such that \( h \restriction M_{j, \alpha} \) is nonzero. Moreover, we have
\[
h \restriction M_{j, \alpha} = h_{j, \alpha} \circ g_{j, \alpha}.
\]
Since \( h \restriction M_{i+1} \) extends \( h_{j, \alpha} \circ g_{j, \alpha} \), and given the way \( M_{i+1} \) was defined, we obtain a contradiction—completing the proof.
\end{proof}

\section {Trivial duality around reduced separable  $p$-groups} \label{1B}The main result
of this section is Theorem \ref{d311}, also  we explore how the class $\mathbf{K}_p$ integrates into the framework developed thus far; see, for example, Corollary \ref{d73} and
\ref{d71}.
  Among separable abelian $p$-groups, the most notable class—besides the class of direct sums of cyclic $p$-groups—is the class of torsion-complete $p$-groups, which we now introduce.

\begin{definition}
	\begin{enumerate}
	\item	Given an abelian group $G$ such that
	$\bigcap_{n\in\mathbb{N}} {nG}=\{0\}$,
	let $G^{\widehat{ }}$ denote its $\mathbb{Z}$-adic completion.	To apply this for $p$-groups, suppose $G$ is such that
	$\bigcap_{n\in\mathbb{N}} {p^nG}=\{0\}$,
		let $ G ^{\widehat{ }_p}$ denote its $p$-adic completion.	\item
		Given $n \in \mathbb{N}$ and a cardinal $\kappa$,  let $B_{n, \kappa}$ denote the group $\bigoplus_{i<\kappa} \frac{\mathbb{Z}}{ p^n \mathbb{Z}}$.
		\item
	By a torsion-complete abelian $p$-group, we mean    the torsion subgroup of the group $(\bigoplus_{n\in\mathbb{N}}B_{n, \kappa_n})^{\widehat{ }_p}$,
		for some sequence $(\kappa_n)_{n \in \mathbb{N}}$ of cardinals.
	\end{enumerate}
\end{definition}

\begin{definition}
\label{d65}
Suppose $p$ is a prime number.
\begin{enumerate}
\item Let $\mathbf d^0_p=\langle  R, M^*, \cM, \theta     \rangle$ be  defined via:
\begin{enumerate}
\item[$(a)$]  $R$ is the ring $\bbZ$ of integers, so an $R$-module is an abelian
  group,

\item[$(b)$]  $M^*$ is the zero $R$-module,

\item[$(c)$]  $\cM =\{\bbZ/p^n \bbZ:n=1,2,\ldots\}$,

\item[$(d)$]  $\theta =\aleph_1$.
\end{enumerate}

\item Let $\Xi^0_p,$ be the class of triples $(G,H,h)$ such that:

\begin{enumerate}
\item[$(a)$]  $G \in \mathbf {K}_p^{[rs]}$ has cardinality $\le
  2^{\aleph_0}$ and is not torsion-complete,

\item[$(b)$]  $H$ has the form $\bigoplus\{G_n:n \in \cU\}$ where $\cU
  \subseteq \bbN$ is infinite and $G_n \cong \bbZ/p^n \mathbb{Z}$,

\item[$(c)$]  $h$ is a non-small homomorphism from $G$ to $H$, or just $H$ has no infinite subgroup which is torsion complete.
\end{enumerate}
\end{enumerate}
\end{definition}

The next routine fact shows that $\mathbf d^0_p$ and $\Xi^0_p$ are indeed as required.
\begin{fact}
\label{d67n}
Let $p$ be prime.
Then  $\mathbf d^0_p$ is a module parameter,
 and $\Xi^0_p$ is a $\mathbf d^0_p$-problem.
\end{fact}

\begin{proposition}
\label{d31}
Let $0<\mathbf{k}<\omega$ be given, and assume
	$J =J^{\bd}_{\aleph_0 }$ is the ideal of bounded subsets of $\omega$. Then
 $(\aleph_0, J)$ freely fits the triple $(\aleph_0, \mathbf d^0_p, \Xi^0_p)$.
\end{proposition}
\begin{proof}
By Fact \ref{d67n}, $\mathbf d^0_p$ is a module parameter,
and $\Xi^0_p$ is a $\mathbf d^0_p$-problem. In particular, clause (A) of Definition \ref{d60n} is satisfied. To check \ref{d60n}(B),
we proceed as follows. First, recall that since the groups under consideration are separable, so the completion operator is one-to-one. Let $(g_n)_{n\in\mathbb{N}}\in \hat{G_*}\subseteq \prod_n \frac{G_*}{p^nG_*}$. The assignment $(g_n)_{n\in\mathbb{N}}\mapsto (h_*(g_n))_{n\in\mathbb{N}}$ induces  a unique extension $\hat{h_*} \in \hom(\hat{G_*}, \hat{H_*})$ of $h_*$: $$\xymatrix{
			&& G_*\ar[d]_{\subseteq_{G_*}}\ar[r]^{h_*}&H_* \ar[d]^{\subseteq_{H_*}}&\\
			&& \hat{G_*}\ar[r]^{\hat{h_*}}& \hat{H_*}
			&&&}$$
		Since $h_\ast$ is not small, $\rng({h_*})\cong G_*/\ker({h_*}) $ is not  finite. But, $H_*$ is countable. Hence $|\rng({h_*})|={\aleph_0}$. We next claim that
$|\rng(\hat{h_*})|=2^{\aleph_0}$.
Indeed, as $h_\ast$ is not small, there are $k<\omega$ and $g_n$ for $n<\omega$ such that
\begin{itemize}
		\item[$-$]  $G_* \models`` p^k(p^ng_n)=0$'',
		
		\item[$-$] 	$ \hat{h_*}(p^ng_n)\neq 0$.
	\end{itemize}	So, given any $\overline{a}:=\langle a_0,\ldots, a_n,\ldots \rangle \in$$ ^{\omega}\mathbb{Z} $ and $m<\omega$, there is a well-defined element $$b_{\overline{a},m}:=\Sigma \{a_np^nh_\ast(g_{n+m}):n<\omega\}\in G_0 :=\rng(\hat{h_*}).$$Then
$\hat{H_*}$ has an infinite subgroup.
Since this group is independent of the choice of $\overline{a}$, we deduce that	$|\rng(\hat{h_*})|\geq 2^{\aleph_0}$. Since the reverse inequality is trivial, we get the desired claim.
In view of last claim, $\rng(\hat h_*) \nsubseteq H_*.$ Let us take  $G_1$ to be any group furnished with the following two properties:
\begin{itemize}
	\item[$(a)$]  $G_* \subseteq G_1 \subseteq_{*} \hat{G_*}$,
	
	\item[$(b)$] 	$\rng(\hat{h_*}\rest G_1) \nsubseteq H_*$.
\end{itemize}

Suppose by the way of contradiction  that there is an $h_1 \in \Hom (G_1 ,H_*)$ such that
such that the following diagram commutes:$$\xymatrix{
	& 0 \ar[r]&G_*\ar[r]^{\subseteq}\ar[d]_{h_*}&G_1\ar[dl]^{\exists h_1}\\
	&& H_*
	&&&}$$
Let us conveniently summarize the results with the following diagram:
$$\xymatrix{
	&& G_1 \ar[d]_{h_1}\ar[r]^{\subseteq_{G_1}}& \hat G_1 \ar[r]^{=}\ar[d]_{\hat h_1}&\hat G_{\ast}\ar[d]^{\hat h_{\ast}}\\
	&& H_{\ast}\ar[r]^{\subseteq_{H_*}}& \hat H_{\ast}\ar[r]^{=}&\hat H_{\ast}
	&&&}$$This is in contradiction with $(b)$. Thus, there is no such $h_1$, proving the frame is fit, as desired.
\end{proof}

\begin{definition}
\label{d65}
Suppose $p$ is a prime number.
\begin{enumerate}
\item Let $\mathbf d^1_p=\langle  R, M^*, \cM, \theta     \rangle$ be  defined via:
\begin{enumerate}
\item[$(a)$]  $R$ is the ring $\bbZ$ of integers,

\item[$(b)$]  $M^*$ is the zero $R$-module,

\item[$(c)$]  $\cM =\{\bbZ/p^n \bbZ:n=1,2,\ldots\}$,

\item[$(d)$]  $\theta =\aleph_1$.
\end{enumerate}

\item Let $\Xi^1_p,$ be the class of triples $(G_*,H_*,h_*)$ such that:

\begin{enumerate}
\item[$(a)$] $G_*$ is of the form
 $$G_*:=\bigoplus\{G^*_{m, \alpha}: m < \bold{k}, \alpha < \omega_1     \},$$
where $G^*_{m,\omega\alpha+n} \cong \mathbb{Z} / p^{n+1}\mathbb{Z}$.

\item[$(b)$]  $H_*:=G_*$,

\item[$(c)$]  $h_*:=\id_{G_*}$.
\end{enumerate}
\end{enumerate}
\end{definition}
\begin{fact}
\label{d67n1}
Let $p$ be prime.
Then $\mathbf d^1_p$ is a module parameter,
and also $\Xi^1_p$ is a $\mathbf d^1_p$-problem.
\end{fact}
\begin{proof}
This is routine.
\end{proof}Now, we are ready to formulate the main result of this section:
\begin{theorem}
\label{d311}
Let $0<\mathbf{k}<\omega$ be given, and assume
	$J =J_{ \aleph_1}^{\bd}$ is the ideal of bounded subsets of $\omega_1$. Then $(\aleph_1, J)$
 fits the triple $(\aleph_1, \mathbf d^1_p, \Xi^1_p)$.
\end{theorem}
\begin{proof}
By Fact \ref{d67n1}, $\mathbf d^1_p$ is a module parameter,
and also $\Xi^1_p$ is a $\mathbf d^1_p$-problem, i.e.,   Definition \ref{d60n}(A) is satisfied. In order to check the property presented in its clause (B), we recall  that
$$G_*\cong \bigoplus\{ \mathbb{Z}x_{m, \alpha}: m < \bar{k}, \alpha < \omega_1            \},$$
where $\mathbb{Z}x_{m, m,\omega\alpha+n}:=\mathbb{Z}/ p^{n+1} \mathbb{Z}$. This is well-defined because   any ordinal less than $\omega_1$ is of the form $\omega\alpha+n$ for unique ordinals  $n<\omega$ and $\alpha$.
In other words, $\ord(x_{m, \omega\alpha+n})=p^{n+1}$.
Toward defining $G_1$, let $\rho_\alpha \in$$^{\omega}\alpha$, for $\alpha < \omega_1$ be increasing and pairwise distinct.
Then $G:=G_1$ will be the abelian group generated by
$$\{x_{m,\alpha}:\alpha < \aleph_1,m < \mathbf k\} \cup
\{y^1_{\rho_\alpha,n}:\alpha < \omega_1,n < \omega\} \cup \{y^2_\rho:
\rho \in {}^{\omega >}2\}$$ freely except the equations:

\begin{enumerate}
	\item[$(*)^1_{\alpha,n}$]  $\quad\quad\quad p y^1_{\rho_\alpha,n+1} = y^1_{\rho_\alpha,n} -
	y^2_{\rho_\alpha \rest n} - \sum\limits_{m < \mathbf k}
	x_{m,\omega \cdot \alpha +n},$
	\item[$(*)^2_{\alpha,n}$]  $\quad\quad\quad p^{n+1} y^1_{\rho_\alpha,n}=0= p^{n+1} y^2_{\rho_\alpha \restriction n}$,
\end{enumerate}
where $\alpha < \omega_1, n<\omega$ with the convenience that $\ord(y^1_{\rho_\alpha,n})=\ord(y^2_{\rho_\alpha \restriction n})=p^{n+1}$. In particular,
$p^{n} y^1_{\rho_\alpha,n}\neq0\neq p^{n} y^2_{\rho_\alpha \restriction n}$.

According to its definition,  $H_*$ and
$G_*$ are free with respect to $\mathbf{K}_p$. Here, we are going to show  $G_* \subseteq G$, and $G$ is $(\aleph_1,\mathbf{K}_p)$-free.
Indeed, let $G' \subseteq G$ be  a countable subgroup. Recall from
$(*)^1_{\alpha,n}$ that $\{y^2_\rho:
\rho \in {}^{\omega >}2\}$ can be drive from other terms.
Then there is an $\alpha < \omega_1$  such that
$$G' \subseteq \langle \{x_{m, \beta}: m < \bold{k}, \beta < \omega\cdot \alpha   \} \cup \{y^1_{\beta, n}: \beta < \omega\cdot \alpha, n<\omega     \}      \rangle_G.$$This gives us a generating set for $G'$. Also, recall that
 the only relations on these generators of $G'$ involved in $\{x_{m, \beta}: m < \bold{k}, \beta < \omega\cdot \alpha   \} \cup \{y^1_{\beta, n}: \beta < \omega\cdot \alpha, n<\omega     \}$ are those coming from $(*)^2_{\alpha,n}$. Combining these, it turns out that $G'$ is $\mathbf{K}_p$-free.
In other words, $G$ is $(\aleph_1,\mathbf{K}_p)$-free.
In the same vein we observe that
$G/G_*$ is $(\aleph_1,\mathbf{K}_p)$-free.

Let $g: \beta < \omega\cdot \alpha \rightarrow \omega$ be such that
$\langle  \rho_\beta \upharpoonright g(\beta): \beta < \omega\cdot \alpha    \rangle$ are pairwise
$\unlhd$-incomparable.
Therefore, things are reduced to showing that
there is no homomorphism $h \in \Hom(G, H_*)$, extending $h_*=\id_{G_*}$.
Suppose by the way of contradiction  that there is an $\hat{h} \in \Hom(G  ,G_*)$ such that the following diagram commutes
$$\xymatrix{
	& 0 \ar[r]&G_*\ar[r]^{\subseteq}\ar[d]_{=}&G\ar[dl]^{ \hat{h}}\\
	&& G_*
	&&&}$$
For each ordinal $\beta$,  we look at  $L_\beta:=\bigoplus\{ \mathbb{Z}x_{m, \beta+n}: m < \bold{k}, n< \omega  \}.$  Then $L_\beta$ is countable, and there is a projection $\pi_\beta$ from $H_*=G_*$ onto
$L_{\beta}.$
For each countable ordinal $\beta$ set $\tilde h_\beta:=\pi_\beta \circ (\hat{h}\rest)
$:$$\xymatrix{
	&     G_*\ar[r]^{\pi_\beta}\ar[d]_{\tilde h_\beta}&L_\beta\ar[dl]^{\hat{h}\rest }\\
 & G_*,
	&&&} $$

Recall that
$L_\beta$ is of countable size,  for these ordinals.
For every  $n< \omega$ and  $\alpha < \omega_1$, we bring the following claim
\begin{enumerate}
	\item[$(\ast)$]$
	y^1_{\alpha, 0}=p^n y^1_{\alpha, n+1}+  \sum_{i \leq n}p^i y^2_{\rho_\alpha \restriction i}+  \sum_{i \leq n}\sum_{m < \bold k} p^i x_{m, \omega\alpha+i}.
$\end{enumerate}
Indeed, we proceed by induction on $n$. For $n=0$ this is clear. Assume it holds for $n$, then we have
\[\begin{array}{ll}
y^1_{\alpha, 0} &= p^{n+1} y^1_{\alpha, n+1}+  \sum_{i\leq n}p^i y^2_{\rho_\alpha \restriction i}+  \sum_{i \leq n}\sum_{m < \bold k} p^i x_{m, \omega\alpha+i} \\
&= p^{n+1} [p y^1_{\alpha, n+2}+y^2_{\rho_\alpha \restriction n+1}+\sum_{m < \bold k}  x_{m, \omega\alpha+n+1}]\\
&\quad+  \sum_{i\leq n}p^i y^2_{\rho_\alpha \restriction i}+  \sum_{\leq n}\sum_{m < \bold k} p^i x_{m, \omega\alpha+i}\\
&=p^{n+2}y^1_{\alpha, n+2} +  \sum_{i \leq n+1}p^i y^2_{\rho_\alpha \restriction i}+  \sum_{i \leq n+1}\sum_{m < \bold k} p^i x_{m, \omega\alpha+i},  \end{array}\]
as claimed by $(\ast)$.

Applying $\tilde h_\beta$ on both sides of $(\ast)$, and noting that $\tilde h_\beta$ is identity on $x_{m, \omega\beta+i}$'s,  we lead to:
\begin{enumerate}
	\item[$(*)_{1}$] For every  ordinal $\beta<\omega_1,$ $n< \omega$
	\[
	\tilde h_\beta(y^1_{\beta , 0})=p^{n+1} \tilde h_\beta(y^1_{\beta , n+1})+  \sum_{i \leq n}p^i \tilde h_\beta(y^2_{\rho_{\beta } \restriction i})+ \sum_{i \leq n}\sum_{m < \bold k} p^i x_{m, \omega\beta+i}.
	\]
\end{enumerate}
Note that for each $\beta<\omega_1,$ there is some $n_\beta < \omega$ such that
\[
\tilde h_\beta(y^1_{\beta , 0}) \in L_\beta \restriction n_\beta = \bigoplus\{ \mathbb{Z}x_{m, \beta+n}: m < \bold{k}, n<  n_\beta  \}.
\]
Thanks to Fodor's lemma, we can find some $n_*< \omega$ such that
the set
\[
S_1= \{ \beta < \omega_1:  n_\beta=n_* \}
\]
is stationary in $\omega_1$. Again, according to Fodor's lemma, we can find some $\rho$ such that the set
\[
S_2=\{\beta \in S_1: \rho_\beta \restriction n_*+1 =\rho  \}
\]
is stationary.  Take some $\beta < \alpha$ in $S_2$ with $\omega\beta=\beta$ and $\omega\alpha=\alpha$.  By $(*)_1$ applied to $\beta, \alpha$ and $m$ we have:
\begin{enumerate}
	\item[$(**)_{1}$] $\tilde h_\beta(y^1_{\beta , 0})=p^{n_*+1} \tilde h_\beta(y^1_{\beta , n_*+1})+  \sum_{i \leq n_*}p^i \tilde h_\beta(y^2_{\rho_{\beta } \restriction i})+ \sum_{i \leq n_*}\sum_{m < \bold k} p^i x_{m, \omega\beta+i}.$

\item[$(**)_{2}$] $\tilde h_\alpha(y^1_{\beta , 0})=p^{n_*+1} \tilde h_\alpha(y^1_{\alpha , n_*+1})+  \sum_{i \leq n_*}p^i \tilde h_\beta(y^2_{\rho_{\alpha } \restriction i})+ \sum_{i \leq n_*}\sum_{m < \bold k} p^i x_{m, \omega\alpha+i}.$
\end{enumerate}
Let us now consider the following two projections:
\begin{itemize}
	\item[] $\pi_{\beta, n_*}: L_\beta \longrightarrow L_\beta (n_*):=\bigoplus\{ \mathbb{Z}x_{m, \beta+n_*}: m < \bold{k} \}$,
	\item[] $\pi_{\alpha, n_*}: L_\beta \rightarrow L_\alpha(n_*)$.
\end{itemize}

Set
$$\tilde{h}:=(\tilde h_\beta \circ \pi_{\beta, n_*}) \oplus (\tilde h_\alpha \circ \pi_{\alpha, n_*}): G \longrightarrow L_{\beta}(n_*) \oplus L_{\alpha}(n_*).$$
Since we know the kernel of the projections, and since  $\hat{h}\rest G_\ast=\id$, we deduce the following:
\begin{itemize}
\item $\tilde{h}(y^1_{\beta , 0})=\tilde{h}(y^1_{\alpha , 0})=0,$
\item $ \tilde h(x_{m, \beta+i})=\tilde h(x_{m, \alpha+i})=0$ for all $i< n_*$,
\item $\tilde h(x_{m, \beta+n_*})=x_{m, \beta+n_*}$ and   $\tilde h(x_{m, \alpha+n_*})=x_{m, \alpha+n_*}$.\end{itemize}
Also, recall that
\begin{itemize}
\item  $\rho_\beta \upharpoonright i =\rho_\alpha \upharpoonright i$ for all $i \leq n_*$.\end{itemize}
Thus, by subtracting the equations $(**)_{1}$ and $(**)_{2}$ and by plugging these bullets, we lead to the following equality viewed  in   $L_{\beta}(n_*) \oplus L_{\alpha}(n_*)$:
\begin{itemize}
\item[$(\dagger)$:]  $p^{n_*+1}\tilde{h}(y^1_{\beta , n_*+1}-y^1_{\alpha , n_*+1}) = -p^{n_*} (x_{m, \beta+n_*}-x_{m, \alpha+n_*}) .$\end{itemize}Recall that
 $\ord(x_{m, \beta+n_*}-x_{m, \alpha+n_*})=p^{n_*+1}.$
Hence $p^{n_*+1}\tilde{h}(y^1_{\beta , n_*+1}-y^1_{\alpha , n_*+1}) \neq 0, $ because the  right hand side of $(\dagger)$ is nonzero. Since $$L_{\alpha}(n_*)=\bigoplus\{ \mathbb{Z}x_{m, \beta+n_*}: m < \bold{k} \},$$so we deduce, from the first paragraph of the current proof, that any of its element is annihilated by $p^{n_\ast+1}$. This implies that
  $$p^{n_*+1}\tilde{h}(y^1_{\beta , n_*+1}-y^1_{\alpha , n_*+1}) = 0,$$
a contradiction, and so
  there is no homomorphism $\hat{h} \in \Hom(G, H_*)$, extending $h_*$.
\end{proof}

Now, we are ready to confirm
Problem  \ref{p1.2} in the following sense:
\begin{corollary}
\label{d71}Let $0<\mathbf{k}<\omega$ be given and let $p$ be a prime number.    There is an abelian group $G$   equipped with the following two properties:

\begin{enumerate}
	\item[$(a)$]  $G$ is   $\aleph_{\omega_1
		\cdot \mathbf k}$-free
	with respect to $\mathbf{K}_p$,
	
	\item[$(b)$]  $\Hom(G,F) = 0$ for all indecomposable $\mathbf {K}_p$-free groups $F$.
\end{enumerate}

\end{corollary}
\begin{proof}
Given $\mathbf{k}$, we use Discussion \ref{exmbb} to find a combinatorial $\bar{\partial}$-parameter $\mathbf{x}$ which is $\aleph_{\omega_1 \cdot \mathbf{k}}$-free and equipped with a $\chi$-black box, where $\chi := |R| + \aleph_1$ and $J :=J_{ \aleph_1}^{\bd}$. In view of Theorem \ref{d311}, $\mathbf{x}$ $\aleph_1$-fits the triple $(\aleph_1, \bold d^1_p, \Xi^1_p)$.
In order to get the desired conclusion, it remains to apply Theorem \ref{d64n}.
\end{proof}

\begin{corollary}\label{d73}
Let $0<\mathbf{k}<\omega$ and let $p$ be a prime number.  Then there is a group $G $ equipped with the following two properties:

\begin{enumerate}
\item[$(a)$]  if $(G_*, H_*, h_*) \in \Xi^0_p$,  then  every $h \in \Hom(G,H_*)$ is small.

\item[$(b)$]  $G$ is $\aleph_{\omega \cdot \mathbf k}$-free with respect to $\mathbf {K}_p $.
\end{enumerate}
\end{corollary}

\begin{proof}
Recall that we can find a combinatorial $\bar{\partial}$-parameter $\mathbf{x}$ which is $\aleph_{\omega \cdot \mathbf{k}}$-free and equipped with a $\chi$-black box, where $\bar{\partial}$ is the constant sequence $\aleph_0$ of length $\bold k$,  $\chi :=  \aleph_1$ and $J :=J_{ \aleph_0}^{\bd}$.
 Thanks to Proposition \ref{d31}, $\mathbf x$ freely $\aleph_1$-fits the triple $(\aleph_0, \bold d^0_p, \Xi^0_p)$.
	 This allows us to use Theorem \ref{d64n} to get the required group $G$.
\end{proof}
\section*{Acknowledgements}
The authors thank the referee of the paper.


\begin{thebibliography}{}


\bibitem{ags2}
M. Asgharzadeh,  M. Golshani  and   S. Shelah,   \emph{Quite free abelian groups  with prescribed endomorphism ring}, work in progress.


\bibitem{eda}
K. Eda, \emph{ Cardinal restrictions for preradicals}, Abelian Group Theory, Contemp. Math.
{\bf{87}} (1989), 277-283.




\bibitem{EM02}
{P. C. Eklof and A. Mekler}, \emph{Almost free modules: Set theoretic methods},
{Revised Edition},
{North--Holland Publishing Co.},
{North--Holland Mathematical Library},
{\bf{65}},
{2002}.




\bibitem{Fu} {L. Fuchs}, \emph{Infinite Abelian Groups}
{New York Academic Press},
{I, II},
{1970, 1973}.




\bibitem{fuchs} L. Fuchs,  \emph{Abelian groups}, Springer Monographs in Mathematics. Springer, Cham, 2015.

\bibitem{Sh:920} R.
G\"obel and  S. Shelah,  \emph{$\aleph_n$-free modules with trivial
	duals},
{Results Math.},
number  {1-2},
{ \bf{54}},
{(2009)},
{53--64}.






\bibitem{GT} R.
G\"obel and   J. Trlifaj,   \emph{Approximations and endomorphism algebras of modules},
Vols. 1, 2, de Gruyter Expositions in Mathematics, Vol. 41, Walter de Gruyter,
Berlin, 2012.








 \bibitem{Sh:44} S. Shelah,  \emph{Infinite abelian groups, Whitehead problem and some constructions}, Israel Journal
of Mathematics {\bf{18}} (1974), 243-256.



\bibitem{Sh:105}
S. Shelah, \emph{On uncountable abelian groups}, Israel J. Math. {\bf{32}} (1979), 311-330.


\bibitem{Sh:227} S. Shelah,   \emph{A combinatorial theorem and endomorphism rings of abelian groups  II}. In Abelian groups and modules (Udine, 1984), Vol. {\bf{287}}, Springer, Vienna,  37–86.

\bibitem{Sh:883}   S. Shelah, \emph{$\aleph_n$-free abelian group with no
	non-zero homomorphism to $\mathbb Z$},
{Cubo},
{2},
{\bf{9}}
{2007} {59--79}.



\bibitem{Sh:898}
 S. Shelah, \emph{Pcf and abelian groups}
 {Forum Math.},
	{\bf{25}}{5},    {967--1038}
{(2013)}.



\bibitem{Sh:1028}
 S. Shelah,
\emph{Quite free complicated abelian group, pcf
			and Black Boxes},
  {Israel Journal of Mathematics}, {\bf{240}} (2020), no. 1, 1-64.




 \end{thebibliography}
\end{document}